\documentclass[11pt,a4paper]{article}

\usepackage{amsmath, amssymb, amsthm} 
\usepackage{thmtools} 
\usepackage{xspace}
\usepackage{enumitem}
\usepackage{hyperref}
\usepackage{mathrsfs}
\usepackage{stackrel}

\allowdisplaybreaks
\expandafter\let\expandafter\oldproof\csname\string\proof\endcsname
\let\oldendproof\endproof
\renewenvironment{proof}[1][\proofname]{%
	\oldproof[\bf #1]%
}{\oldendproof}

\allowdisplaybreaks

\theoremstyle{plain}
\newtheorem{theorem}{Theorem}[section]
\newtheorem{lemma}[theorem]{Lemma}
\newtheorem{claim}[theorem]{Claim}
\newtheorem{proposition}[theorem]{Proposition}

\newtheorem{conjecture}[theorem]{Conjecture}

\newcommand{\sm}{\setminus}

\renewcommand{\Pr}{\mathbb{P}}

\makeatletter
\def\moverlay{\mathpalette\mov@rlay}
\def\mov@rlay#1#2{\leavevmode\vtop{%
   \baselineskip\z@skip \lineskiplimit-\maxdimen
   \ialign{\hfil$\m@th#1##$\hfil\cr#2\crcr}}}
\newcommand{\charfusion}[3][\mathord]{
    #1{\ifx#1\mathop\vphantom{#2}\fi
        \mathpalette\mov@rlay{#2\cr#3}
      }
    \ifx#1\mathop\expandafter\displaylimits\fi}
\makeatother

\newcommand{\discup}{\charfusion[\mathbin]{\cup}{\cdot}}

\let\eps=\varepsilon
\let\theta=\vartheta
\let\rho=\varrho
\let\phi=\varphi

\usepackage{framed}

\makeatletter
\renewcommand*{\eqref}[1]{%
  \hyperref[{#1}]{\textup{\tagform@{\ref*{#1}}}}%
}
\makeatother

\newcommand\G{\mathcal{G}}

\newcommand{\Bin}{\ensuremath{\textrm{Bin}}}

\title{Rainbow trees in uniformly edge-coloured graphs}

\author{Elad Aigner-Horev \thanks{Department of Computer Science, Ariel University, Ariel 40700, Israel. Email: {\tt horev@ariel.ac.il}.}
\quad Dan Hefetz \thanks{Department of Computer Science, Ariel University, Ariel 40700, Israel. Email: {\tt danhe@ariel.ac.il}. Research supported by ISF grant 822/18.}
\quad Abhiruk Lahiri \thanks{Department of Computer Science, Ariel University, Ariel 40700, Israel. Email: {\tt abhiruk@ariel.ac.il}.}}

\begin{document}
\maketitle

\begin{abstract}
We obtain sufficient conditions for the emergence of spanning and almost-spanning bounded-degree {\sl rainbow} trees in various host graphs, having their edges coloured independently and uniformly at random, using a predetermined palette. Our first result asserts that a uniform colouring of $\mathbb{G}(n,\omega(1)/n)$, using a palette of size $n$, a.a.s. admits a rainbow copy of any given bounded-degree tree on at most $(1-\eps)n$ vertices, where $\eps > 0$ is arbitrarily small yet fixed. This serves as a rainbow variant of a classical result by Alon, Krivelevich, and Sudakov pertaining to the embedding of bounded-degree almost-spanning prescribed trees in $\mathbb{G}(n,C/n)$, where $C > 0$ is independent of $n$. 

Given an $n$-vertex graph $G$ with minimum degree at least $\delta n$, where $\delta > 0$ is fixed, we use our aforementioned result in order to prove that a uniform colouring of the randomly perturbed graph $G \cup \mathbb{G}(n,\omega(1)/n)$, using $(1+\alpha)n$ colours, where $\alpha > 0$ is arbitrarily small yet fixed, a.a.s. admits a rainbow copy of any given bounded-degree {\sl spanning} tree. This can be viewed as a rainbow variant of a result by Krivelevich, Kwan, and Sudakov who proved that $G \cup \mathbb{G}(n,C/n)$, where $C > 0$ is independent of $n$, a.a.s. admits a copy of any given bounded-degree spanning tree. 

Finally, and with $G$ as above, we prove that a uniform colouring of $G \cup \mathbb{G}(n,\omega(n^{-2}))$ using $n-1$ colours a.a.s. admits a rainbow spanning tree. Put another way, the trivial lower bound on the size of the palette required for supporting a rainbow spanning tree is also sufficient, essentially as soon as the random perturbation a.a.s. has edges. 
\end{abstract}

\section{Introduction}

Given a graph $G$ and an edge-colouring $\psi$ of $G$, a subgraph $H \subseteq G$ is said to be $\psi$-{\em rainbow} if 
$|\{\psi(e) : e \in E(H)\}| = e(H)$ holds; i.e.\ if $\psi(e) \neq \psi(e')$ for every two distinct edges $e, e' \in E(H)$. An edge colouring of a graph $G$ is said to be $k$-{\em uniform}, if the colour of every edge of $G$ is chosen independently and uniformly at random from a prescribed colour palette of size $k \in \mathbb{N}$. 

We obtain sufficient conditions for the emergence of bounded-degree rainbow trees in various host graphs. First, we consider the rainbow emergence of {\sl prescribed almost-spanning} bounded-degree trees in uniformly coloured (sparse binomial) random graphs; establishing a rainbow variant for the classical result of Alon, Krivelevich, and Sudakov~\cite{AKS07}. Second, we consider the rainbow emergence of {\sl prescribed spanning} bounded-degree trees in uniformly coloured graphs that consist of a dense \emph{seed} which is perturbed by a (sparse binomial) random graph; establishing a rainbow variant for a result of Krivelevich, Kwan and Sudakov~\cite{KKS17}. Finally, we study the emergence of rainbow spanning trees in uniformly coloured randomly perturbed dense graphs; here the tree is not prescribed. 

\medskip

The study of the emergence of {\sl almost-spanning} trees with bounded degrees in $\mathbb{G}(n,p)$ dates back to a conjecture put forth by Erd\H{o}s~\cite{Erdos}, stipulating that $G \sim \mathbb{G}(n,c/n)$ a.a.s. contains a path of length at least $(1-\alpha(c))n$, with $\lim_{c \to \infty}\alpha(c) = 0$. Following an earlier result by Fernandez de la Vega~\cite{Vega}, proving a slightly weaker version of this conjecture, the conjecture was resolved by Ajtai, Koml\'os, and Szemeredi~\cite{AKSz}. Refinements regarding $\alpha(c)$ and its rate of convergence were provided by Bollob\'as~\cite{Bollobas} and by Frieze~\cite{F86}. The study of the emergence of almost-spanning trees with bounded degrees (rather than just paths) in random graphs was initiated by Fernandez de la Vega~\cite{Vega88}. Significant progress in this line of research was attained by Alon, Krivelevich, and Sudakov~\cite{AKS07} who proved that {\sl for every integer $d \geq 2$ and every $\eps > 0$ there exists a constant $c > 0$ such that 
$\mathbb{G}(n,c/n)$ a.a.s. contains every tree on $(1-\eps)n$ vertices whose maximum degree is at most $d$.} The upper bound on the constant $c$ was subsequently improved by Balogh, Csaba, Pei, and Samotij~\cite{BCPS10}.

The emergence of rainbow Hamilton cycles in uniformly coloured (binomial) random graphs has been studied extensively; in particular, in~\cite{BF16,F15,FK16,FL} one encounters a series of improvements and refinements regarding the density of the random graph being coloured as well as the size of the colour palette being used. 

The study of the emergence of various spanning configurations in randomly perturbed (hyper)graphs dates back to the work of Bohman, Frieze, and Martin~\cite{BFM03}, and received significant attention in recent years, see e.g.\ ~\cite{BTW17, BHKM18, BFKM04, BHKMPP18, BMPP18, DRRS18, HMMMO, HZ18, KKS16, KKS17, MM18} to name just a few. In most results in this area, the random perturbation is the binomial random graph $\mathbb{G}(n,p)$ (or the, essentially equivalent, Erd\H{o}s-R\'enyi random graph $\mathbb{G}(n,m)$); however, quite recently 
other distributions have been considered, such as random geometric graphs~\cite{Diaz} and random regular graphs~\cite{DG21}. In the context of rainbow spanning configurations, the emergence of rainbow Hamilton cycles in uniformly coloured randomly perturbed graphs was studied by Anastos and Frieze~\cite{AF} as well as by the first two authors~\cite{US}.

\subsection{Our results}
 
We write $(G, \psi)$ to denote a graph $G$ whose edges are coloured according to a colouring $\psi$. Our first main result can be viewed as a rainbow variant of the aforementioned result of~\cite{AKS07}.

\begin{theorem}\label{thm:rainbow-AKS}
Let $\eps > 0$ be a real number, let $d \geq 2$ be an integer, and let $T$ be a tree on at most $(1-\eps)n$ vertices and with maximum degree at most $d$. Then, $(G,\psi)$ a.a.s. admits a $\psi$-rainbow copy of $T$, where $G \sim \mathbb{G}(n,\omega(1)/n)$ and $\psi$ is an $n$-uniform colouring of $G$. 
\end{theorem}

The corresponding result of~\cite{AKS07} handles the embedding of almost-spanning prescribed trees of bounded-degree in $\mathbb{G}(n,C/n)$, where $C$ is an explicit constant. We conjecture that the analogous rainbow variant holds as well.

\begin{conjecture}\label{con:rainbow-AKS}
For every $\eps > 0$ and every integer $d \geq 2$, there exists a constant $C > 0$ such that the following holds. Let $T$ be a tree on at most $(1-\eps)n$ vertices with maximum degree at most $d$. Then, $(G,\psi)$ a.a.s. admits a $\psi$-rainbow copy of $T$, where $G \sim \mathbb{G}(n, C/n)$ and $\psi$ is an $n$-uniform colouring of $G$. 
\end{conjecture}

Our second main result pertains to rainbow embeddings of prescribed {\sl spanning} trees with bounded maximum degree in uniformly coloured randomly perturbed dense graphs. The latter have the form $G \cup R$, where $G$ is a fixed graph taken from some graph family $\mathcal{F}$ of interest, and $R$ is a graph on the same vertex set as $G$, sampled from a prescribed graph distribution.

In this paper, we consider a single model of randomly perturbed graphs. For a fixed $\delta > 0$, let $\G_{\delta,n}$ denote the set of $n$-vertex graphs with minimum degree at least $\delta n$. By a random {\em perturbation} of a graph $G \in \G_{\delta,n}$, we mean the distribution $G \cup \mathbb{G}(n,p)$ set over the supergraphs of $G$. 
Given a graph property $\mathcal{P}$, we say that $\mathcal{G}_{\delta,n} \cup \mathbb{G}(n,p)$ a.a.s. satisfies $\mathcal{P}$ if for {\sl every} $G \in \mathcal{G}_{\delta,n}$ it holds that $G \cup \mathbb{G}(n,p)$ satisfies $\mathcal{P}$ asymptotically almost surely. We emphasise that $\mathcal{G}_{\delta,n} \cup \mathbb{G}(n,p)$ is mere notation. In the sequel, we write $\Gamma \sim \mathcal{G}_{\delta,n} \cup \mathbb{G}(n,p)$ and refer to $\Gamma$ as a {\em randomly perturbed} graph with the understanding that the underlying dense {\sl seed}, taken from $\mathcal{G}_{\delta,n}$, is arbitrary yet fixed. 

%We say that $\mathcal{G}_{\delta,n} \cup \mathbb{G}(n,p)$ a.a.s. does {\sl not} satisfy a prescribed property $\mathcal{P}$ if there {\sl exists} a graph $G \in \mathcal{G}_{\delta,n}$ such that $G \cup \mathbb{G}(n,p)$ a.a.s. does not satisfy $\mathcal{P}$. 

\medskip

Krivelevich, Kwan, and Sudakov~\cite{KKS17} proved that for every real $\delta >0$ and positive integer $d$, there exists a $C := C(\delta,d) >0$ such that $\mathcal{G}_{\delta,n} \cup \mathbb{G}(n,C/n)$ a.a.s. contains any given spanning tree with maximum degree at most $d$. Our second main result can be viewed as a rainbow variant of their result.

\begin{theorem} \label{th::rainbowTree}
Let $\delta, \alpha > 0$ be real numbers, let $d \geq 2$ be an integer, and let $T$ be an $n$-vertex tree with maximum degree at most $d$. Then, $(\Gamma,\psi)$ a.a.s. admits a $\psi$-rainbow copy of $T$, where $\Gamma \sim \mathcal{G}_{\delta,n} \cup \mathbb{G}(n,\omega(1)/n)$ and $\psi$ is a $(1+\alpha)n$-uniform colouring of $\Gamma$. 
\end{theorem}

\noindent
Our proof of Theorem~\ref{th::rainbowTree} employs the {\sl absorption} method (see Section~\ref{sec:perturbed}) and is aided by our first main result, namely Theorem~\ref{thm:rainbow-AKS}. Similarly to Conjecture~\ref{con:rainbow-AKS}, we conjecture that a slightly sparser random graph suffices.

\begin{conjecture} \label{con::rainbowTree}
For every $\delta, \alpha > 0$ and integer $d \geq 2$, there exists a constant $C > 0$ such that the following holds. Let $T$ be an $n$-vertex tree with maximum degree at most $d$. Then, $(\Gamma,\psi)$ a.a.s. admits a $\psi$-rainbow copy of $T$, where $\Gamma \sim \mathcal{G}_{\delta,n} \cup \mathbb{G}(n, C/n)$ and $\psi$ is a $(1+\alpha)n$-uniform colouring of $\Gamma$. 
\end{conjecture}

It follows from our proof of Theorem~\ref{th::rainbowTree} that Conjecture~\ref{con::rainbowTree} is a direct consequence of the assertion of Conjecture~\ref{con:rainbow-AKS}. However, perhaps the former is easier to prove than the latter.

\medskip

%Our next result pertains to anti-Ramsey properties of $\mathcal{G}_{\delta,n} \cup \mathbb{G}(n,p)$ with respect to spanning trees. An edge-colouring of a graph $G$ is said to be \emph{$b$-bounded}, if every colour is used on at most $b$ edges. 

%Given $b$, we write $\mathcal{S}_{b,n}$ to denote the set of $n$-vertex graphs having the property that every $b$-bounded colouring thereof admits a rainbow spanning tree. We write $\mathcal{S}_n$ to denote the set of $n$-vertex graphs having the property that every $b$-bounded colouring thereof admits a rainbow spanning tree. 

%\begin{theorem} \label{th::rainbowSpanningTree}\ 
%\begin{description}
%\item [(a)] Let $\delta \in (0,1)$ be fixed and let $b := b(n) \leq \delta^2 n/32$. Then, $\Gamma \sim \mathcal{G}_{\delta,n} \cup \mathbb{G}(n,p)$ a.a.s. has the property that every $b$-bounded colouring of $\Gamma$ admits a rainbow spanning tree, whenever $p:=p(n) = \max \{\omega(n^{-2}), 128 b/(\delta n)^2\}$.  
%
%\item [(b)] Let $\delta \in (0,1/2)$ be fixed and let $p:= p(n) = o(n^{- 2})$. Then, $\mathcal{G}_{\delta,n} \cup \mathbb{G}(n,p)$ a.a.s. has the property that no colouring of its edges admits a rainbow spanning tree. 
%
%\item [(c)] Let $\delta \in (0,1/3)$ be fixed, let $b = \omega(1)$, and let $p := p(n) \leq \leq b/n^2$. Then, $\Gamma \sim \mathcal{G}_{\delta,n} \cup \mathbb{G}(n,p)$ a.a.s. has the property that there is a $b$-bounded colouring of its edges admitting no rainbow spanning tree.
%\end{description}
%\end{theorem}
Any edge-colouring of an $n$-vertex graph containing a rainbow spanning tree requires a colour palette of size at least $n-1$. Theorem~\ref{th::rainbowTree}, dealing with the rainbow embedding of {\sl prescribed} spanning trees, constrains the size of the palette to be $(1+\alpha)n$, with $\alpha > 0$ being arbitrarily small yet fixed. Dropping the requirement that the tree is {\sl predetermined} has the effect that the aforementioned trivial yet necessary lower bound on the size of the palette becomes also sufficient, even with a {\sl minimal} perturbation, so to speak. Our last result reads as follows.

\begin{proposition} \label{th::uarColouring}
Let $\delta \in (0,1)$ be fixed and let $p := p(n) = \omega(n^{-2})$. Then, $(\Gamma,\psi)$ a.a.s. admits a $\psi$-rainbow spanning tree, where $\Gamma \sim \mathcal{G}_{\delta,n} \cup \mathbb{G}(n,p)$ and $\psi$ is an $(n-1)$-uniform colouring of $\Gamma$.
\end{proposition}

We note that a related result for random graphs was obtained by Frieze and McKay~\cite{FM94} (in fact, they proved a stronger hitting-time result). A variation of their result was later proved by Bal, Bennett, Frieze, and Pra\l at~\cite{BBFP15}. Another related result was recently obtained by Bradshaw~\cite{Bradshaw}. 

%\begin{remark}
%It readily follows from Theorem~\ref{th::rainbowSpanningTree}(b) that the bound on $p$ in Theorem~\ref{th::uarColouring} is asymptotically best possible.
%\end{remark}

\section{Prescribed almost-spanning rainbow trees in random graphs}\label{sec:AKS}

In this section, we prove Theorem~\ref{thm:rainbow-AKS}. A set of results facilitating our proof of the latter are collected in Section~\ref{sec:aux-res-AKS}, whereas the proof itself is included in Section~\ref{sec:RBW-AKS-proof}.  

\subsection{Auxiliary results}\label{sec:aux-res-AKS}

Given a linearly-sized prescribed set of colours, taken from the palette of an $n$-uniform colouring of an appropriate random (binomial) graph, the following result provides lower bounds on the number of colours that are a.a.s. used from the prescribed set.  

\begin{lemma}\label{lem:many-colours}
Let $\alpha, \beta, \gamma > 0$ be real numbers and let a positive integer $d$ be fixed. Let $A \subseteq [n]$ be a set of size $\alpha n$. Let $G \sim \mathbb{G}(\beta n, \omega(1)/n)$, let $u \in V(G)$ be arbitrary, and let $\psi : E(G) \to [n]$ be an $n$-uniform colouring. Then, the following hold asymptotically almost surely.
\begin{itemize}
\item [{\em (a)}] $|\psi(E(G)) \cap A| \geq (1 - \gamma) \alpha n$;

\item [{\em (b)}] $|\{\psi(uv) : v \in N_G(u)\} \cap A| \geq d$.
\end{itemize} 
\end{lemma}

\begin{proof}
Starting with (a), note that a.a.s. $e(G) = \omega(n)$. If $|\psi(E(G)) \cap A| < (1 - \gamma) \alpha n$, then there exists a set $B \subseteq A$ of size $\gamma \alpha n$ such that $\psi(E(G)) \cap B = \emptyset$. The probability of this latter event is at most
$$
\binom{\alpha n}{\gamma \alpha n} \left(1 - \frac{\gamma \alpha n}{n} \right)^{e(G)} \leq 2^n e^{- \omega(n)} = o(1).
$$

Next we prove (b); we may clearly assume that $\omega(1) \ll \ln n$. It then follows that a.a.s. $\deg_G(u) \leq \ln n$. Therefore, the probability that $\psi(uv) = \psi(uw)$ for any two distinct vertices $v, w \in N_G(u)$ is at most 
$$
\binom{\ln n}{2} \frac{1}{n} = o(1).
$$   
We may thus condition on the event $|\{\psi(uv) : v \in N_G(u)\}| = \deg_G(u)$, which in turn implies that $X_u := |\{\psi(uv) : v \in N_G(u)\} \cap A|$ is a binomial random variable with parameters $n-1$ and $\frac{\omega(1)}{n} \cdot \frac{|A|}{n} = \frac{\omega(1)}{n}$. Applying Chernoff's inequality then yields
$$
\mathbb{P}(X_u < d) \leq \mathbb{P}(X_u < \mathbb{E}(X_u)/2) \leq e^{- \omega(1)} = o(1). 
$$  
\end{proof}

For a real number $\eta > 0$ and a positive integer $r$, an $n$-vertex graph is said to be an $(\eta,r)$-{\em expander} if 
$|\Gamma_G(X)| \geq r|X|$ for every $X \subseteq V(G)$ of size $|X| \leq \eta n$, where 
$$
\Gamma_G(X) := \{u \in V(G) \setminus X: N_G(u) \cap X \neq \emptyset\}
$$ 
denotes the set of {\em external} neighbours of $X$. 

The following result asserts that attaching a high-degree vertex to an expander, yields an expander (with slightly degraded expansion parameters).

%The setting of the next result is that of {\sl attaching} a single vertex to an appropriately vertex-expanding graph. Lemma~\ref{lem:exp-degrade} provides some control over the degradation in the expansion parameters incurred by such an attachment. 

\begin{lemma}\label{lem:exp-degrade}
For all positive integers $k$ and $d$, there exists an integer $n_0$ such that the following holds for every $n \geq n_0$. Let $G$ be a graph on $n$ vertices and let $u \in V(G)$ be a vertex of degree $\deg_G(u) \geq (d+2)^2$. Suppose that every vertex-induced subgraph of $G \setminus u$ with minimum degree at least $k$ is a $(\frac{1}{2d+1}, d+2)$-expander. Then, every vertex-induced subgraph of $G$ with minimum degree at least $k$ is a $(\frac{1}{2d+2}, d+1)$-expander.
\end{lemma}

\begin{proof}
Fix some vertex-induced subgraph $H$ of $G$ with minimum degree at least $k$ and some set $A \subseteq V(H)$ of size $|A| \leq \frac{V(H)}{2d+2}$. If $u \notin A$, then $|\Gamma_G(A)| \geq |\Gamma_{G \setminus u}(A)| \geq (d+1) |A|$ by assumption as $|A| \leq \frac{V(H)}{2d+2} \leq \frac{V(H \setminus u)}{2d+1}$. Assume then that $u \in A$. If $|A| \leq d+1$, then $|\Gamma_G(A)| \geq \deg_G(u) - |A| \geq (d+2)^2 - (d+1) \geq (d+1) |A|$. Otherwise $|\Gamma_G(A)| \geq |\Gamma_{G \setminus u}(A \setminus u)| \geq (d+2) (|A| - 1) \geq (d+1) |A|$.  
\end{proof}

%We also require the following adaptation of~\cite[Corollary~4.3]{AKS07}, asserting that the tree $T$, per Theorem~\ref{thm:rainbow-AKS}, can be decomposed in a certain advantageous manner.

Next, we collect several results from~\cite{AKS07} which are relevant to our proof. All results are adjusted to our setting, but the proofs are as in~\cite{AKS07} mutatis mutandis and are therefore omitted. Slightly more significant changes are made in the proof of Lemma~\ref{lem:AKS-expander}; these changes are thus described in Appendix~\ref{sec:app-AKS-expand}.

The first of these results allows one to decompose a bounded-degree tree into subtrees whose number is independent of the size of the tree being decomposed. The resulting subtrees are linked to one another in a certain advantageous manner. The decomposition also provides a certain level of control over the sizes of the resulting subtrees.

\begin{proposition}\label{prop:tree-cut}{\em~\cite[Corollary~4.3 -- adapted]{AKS07}} 
Let $\eps \in (0,1/2)$, $\xi > 0$, and an integer $d \geq 2$ be fixed. Let $T$ be a tree on $(1-\eps)n$ vertices with maximum degree at most $d$. Then, $T$ can be decomposed into $s := s_{\text{\ref{prop:tree-cut}}}(d,\eps,\xi)$ subtrees, namely $T_1, \ldots, T_s$, such that for each $2 \leq i \leq s$, the tree $T_i$ is connected to $\bigcup_{j < i} T_j$ via a unique edge; moreover, $\xi n /d \leq v(T_i) \leq \xi n$ holds for every $2 \leq i \leq s$ and $v(T_1) \leq \xi n$. 
\end{proposition}

%\noindent
%{\sl Remark.} Proposition~\ref{prop:tree-cut} can be proved by following the argument seen in~\cite[Corollary~4.3]{3} which repeatedly appeals to~\cite[Proposition~4.2]{3} in order to {\sl peel}, so to speak, the next tree in the decomposition. 
%Assuming a sufficient number of vertices remain .......   

For positive integers $n$ and $r$, and real numbers $\theta, \eta, C > 0$, write $\mathrm{EXPAND}(n,\theta,C,\eta,r)$ to denote the family of $n$-vertex graphs $H$ for which there exists a subgraph $H' \subseteq H$ satisfying the following properties. 
\begin{itemize}
	\item [1.] $v(H') \geq (1-\theta)n$; 
	\item [2.] $C \leq \deg_{H'}(v) \leq 10 C$ for every $v \in 
	V(H')$; and
	\item [3.] every vertex-induced subgraph $H'' \subseteq H'$ 
	with minimum degree $\delta(H'') \geq \ell_1(r,C) := 2 e^4 r^2 \ln(C)$ is an $(\eta,r)$-expander. 
\end{itemize} 
We refer to $H'$ as the {\em effective expander} of $H$.

\begin{lemma}\label{lem:AKS-expander}{\em~\cite[Lemma~3.1 -- adapted]{AKS07}}
For every $\theta \in (0,1/2)$, integer $r \geq 3$, real $0 < \eta \leq (r+2)^{-1}$, and $C$ satisfying $C\geq 50/\theta$ as well as $C \geq \ell_1(r,C)$, the random graph $G \sim \mathbb{G}(n, 4C/n)$ is a.a.s. in $\mathrm{EXPAND}\left(n,\theta,C,\eta,r\right)$.
\end{lemma}

%\noindent
%{\sl Remark.} The bound on the expansion parameter $\eta$ set in Lemma~\ref{lem:AKS-expander} is roughly twice as large as that seen in~\cite[Lemma~3.1]{AKS07}. Lemma~\ref{lem:AKS-expander} is proved in Appendix~\ref{sec:app-AKS-expand}.

\bigskip

For a real number $\eta > 0$ and positive integers $k$ and $d$, set $\ell_2(\eta,d,k) = \frac{\eta k}{40d^2 \ln(2/\eta)}$. The following is the main result of~\cite{AKS07}.

\begin{theorem}\label{thm:AKS-rooted} {\em~\cite[Theorem~1.4 --- adapted]{AKS07}}
For every integer $d \geq 2$ and real $ \eta \in (0,1/2)$, there exists an $n_0 := n_0(\eta,d)$ such that the following holds for all $n \geq n_0$. Let $G$ be an $n$-vertex graph satisfying $\Delta(G) \leq 10\delta(G)$ as well as having the property that every vertex-induced subgraph $H \subseteq G$ with minimum degree at least $\ell_2(\eta,d,\delta(G))$ is a $\left(\frac{1}{2d+2},d+1\right)$-expander. Let $v \in V(G)$ be arbitrary, and let $T$ be a tree on at most $(1-\eta)n$ vertices having maximum degree at most $d$. Then, $G$ contains a copy of $T$ rooted at $v$.  
\end{theorem}

%\noindent
%{\sl Remark.} The fact that a copy of $T$ can be embedded in a {\sl rooted} manner, with the root pre-chosen and arbitrary, is implicit in the proof of~\cite[Theorem~1.4]{AKS07}. 

\subsection{Proof of Theorem~\ref{thm:rainbow-AKS}}\label{sec:RBW-AKS-proof}

Given $d, \eps$, and $T$ as in the premise of the theorem, set $\zeta > 0$ such that 
\begin{equation}\label{eq:zeta}
\zeta \leq \frac{\eps}{2(1-\eps)}.
\end{equation}
Fix $\beta > 0$ and $\rho > 0$ such that\footnote{for constants $a$ and $b$ we write $a \ll b$ to mean that $a$ is sufficiently small with respect to $b$.} 
\begin{equation}\label{eq:eta-beta}
\beta \ll \frac{\zeta \eps}{d^4 \ln(\zeta^{-1})}
\end{equation}
and
\begin{equation}\label{eq:rho}
\rho \ll \eps.
\end{equation}
Finally, write $s := s_{\text{\ref{prop:tree-cut}}}(d,\eps,(1-3\zeta/2)\beta)$.

By Proposition~\ref{prop:tree-cut}, the tree $T$ may be decomposed into $s$ subtrees, namely $T'_1, \ldots, T'_s$, such that $\frac{(1-3\zeta/2) \beta}{d} n  \leq v(T'_i) \leq (1-3\zeta/2)\beta n$ for every $2 \leq i \leq s$ and $v(T'_1) \leq (1-3\zeta/2) \beta n$. Moreover, for every $2 \leq i \leq s$, the tree $T'_i$ is connected to $\bigcup_{j < i} T'_j$ via a unique edge. In particular, for every $i \in [s-1]$, there exists a set $Z_i \subseteq V(T)$ such that 
\begin{itemize}
\item [(T.1)] $|Z_i \cap V(T'_j)| \leq 1$ for every $j \in [s]$; in particular, $|Z_i| \leq s$.

\item [(T.2)] $Z_i \cap V(T'_j) = \emptyset$ for every $j \in [i]$.

\item [(T.3)] For every $x \in Z_i$, there exists a unique vertex $y \in V(T'_i)$ such that $xy \in E(T)$.
\end{itemize}
We refer to the vertices of $Z_i$ as the {\sl roots} of the trees to which $T'_i$ connects. For $i \in [s-1]$, define $T_i$ to be the subtree of $T$ induced by $V(T'_i) \cup Z_i$. This defines the sequence of trees $T_1, \ldots, T_s := T'_s$. Note that $v(T_i) \leq v(T'_i) + s$ holds by (T.1) for every $1 \leq i \leq s$.

Let $\mathcal{R} \subseteq[n]$ be an arbitrary (yet fixed) set of size $\rho n$; the set $\mathcal{R}$ is referred to as the {\sl colour-reservoir}; in the sequel, this set is used for embedding roots in a rainbow fashion.

\bigskip 

Let $G \sim \mathbb{G}(n,\omega(1)/n)$ and let $\psi : E(G) \to [n]$ be an $n$-uniform colouring of $G$. We embed the trees $T_1, \ldots, T_s$ in $G$ one after the other in a rainbow fashion with respect to $\psi$, using a probabilistic embedding procedure; we prove that this procedure terminates a.a.s. with a rainbow copy of $T$ in $G$. For every $1 \leq i \leq s$, let $\mathcal{T}_i$ denote the image of $T_i$ in our embedding of $T$ in $G$. We ensure deterministically that $\bigcup_{i=1}^s \mathcal{T}_i \cong T$ and that $\psi(E(\mathcal{T}_i)) \cap \psi(E(\mathcal{T}_j)) = \emptyset$ holds for every $1 \leq i \neq j \leq s$. Since, moreover, $s$ is fixed, it suffices to prove that the rainbow embedding of each single tree $T_i$ is 
successful asymptotically almost surely. 

Let $X_1,\ldots,X_s \subseteq [n]$ be pairwise-disjoint sets satisfying $|X_i| = (1+3\zeta/2)v(T_i)$ for every $i \in [s]$. For sufficiently large $n$, such a collection of sets exists, since 
$$
\sum_{j=1}^s |X_j| = \sum_{j=1}^s (1 + 3 \zeta/2) v(T_j) \leq (1 + 3 \zeta/2) ((1-\eps) n + s) \overset{\eqref{eq:zeta}}{<} n.
$$
The exposure of $G$ is carried out through $s$ rounds, where (apart from its root, unless $i=1$) $T_i$ is to be embedded in $G[X_i]$ for every $i \in [s]$.

\bigskip

\noindent
{\bf Embedding $\boldsymbol{T_1}$.} Let $X_1 \subseteq [n]$ be an arbitrary  set of size 
\begin{equation} \label{eq::xi1}
\xi_1 n := (1 + 3\zeta/2) v(T_1) 
\end{equation}
so that $\xi_1 \leq (1 + 3\zeta/2) [(1 - 3\zeta/2) \beta + s/n] \leq \beta$ for sufficiently large $n$. Note that $G[X_1] \sim \mathbb{G}(\xi_1 n, \omega(1)/n)$. Let $\mathcal{A}_1 := [n] \setminus \mathcal{R}$ denote the set of colours available for the rainbow embedding of $T_1$. Observe that $|\mathcal{A}_1| \geq (1-\rho)n \overset{\eqref{eq:rho}}{\geq} n/2$. Let $\mathcal{E}_1$ denote the event that an $n$-uniform colouring of $G[X_1]$ uses at least $n/4$ colours from $\mathcal{A}_1$; such an event occurs a.a.s. by Part~(a) of Lemma~\ref{lem:many-colours}. Conditioning on $\mathcal{E}_1$, we perform a random sparsification procedure of (the yet unexposed) $G[X_1]$ as follows.
\begin{itemize}
	\item [(S.1)] Include each $e \in \binom{X_1}{2}$ as an edge 
	independently at random with probability $\omega(1)/
	n$. 

	\item [(S.2)] If $e \in \binom{X_1}{2}$ was included as an 
	edge in Step~(S.1), assign $e$ a colour from $[n]$ uniformly at random. 

	\item [(S.3)] For every $c \in \mathcal{A}_1$, let $E_c$ be the set 
	of edges coloured $c$. 
	If $E_c \neq \emptyset$, choose a single member of $E_c$ uniformly at random, i.e. 
	with probability $1/|E_c|$. 

	\item [(S.4)] Discard all edges not chosen in Step (S.3) (this includes, in particular, all edges whose colour is in $[n] \setminus \mathcal{A}_1$).
	
	\item [(S.5)] From the remaining set of edges choose a subset 
	of size $n/4$ uniformly at random. 
\end{itemize}
The resulting graph distribution, which we denote by $\mathbb{\tilde G}_1$, coincides with the {\sl uniform} graph distribution, namely $\mathbb{G}(\xi_1 n, n/4)$, set over all graphs with $\xi_1 n$ vertices and $n/4$ edges. Moreover, all graphs sampled from $\mathbb{\tilde G}_1$ are by definition rainbow with all colours found in $\mathcal{A}_1$. 

Let $\mathcal{E}_2$ denote the event that $H \sim \mathbb{\tilde G}_1$ is in 
$$
\mathrm{EXPAND}_1  := \mathrm{EXPAND}(\xi_1 n, \zeta/2,  (4\xi_1)^{-2}, (2d+2)^{-1}, d+1),
$$ 
and let $\mathcal{E}_3$ denote the event that $K \sim \mathbb{G}\left(\xi_1 n, \frac{1}{(2\xi_1)^2 n}\right)$ is in $\mathrm{EXPAND}_1$. 
Since $\xi_1 \leq \beta \overset{\eqref{eq:eta-beta}}{\ll} \zeta/d^2$, it follows that $(4\xi_1)^{-2} \geq \max\{50 \zeta^{-1}, \ell_1(d+1,(4\xi_1)^{-2})\}$, implying that $\mathcal{E}_3$ occurs a.a.s., by Lemma~\ref{lem:AKS-expander}. As $\mathbb{\tilde G}_1$ coincides with $\mathbb{G}(\xi_1 n, n/4)$, the event $\mathcal{E}_2$ holds a.a.s. as well due to the well-known relation between $\mathbb{G}(n,p)$ and $\mathbb{G}(n,m)$ with appropriate parameters (see, e.g.\ ~\cite[Proposition~1.13]{JLR}).

Fix a graph $H \sim \mathbb{\tilde G}_1$ satisfying $\mathrm{EXPAND}_1$ and let $H' \subseteq H$ be its effective expander; in particular, $v(H') \geq |X_i| - \zeta |X_i| /2 \overset{\eqref{eq::xi1}}{\geq} (1 + 3\zeta/4) v(T_1)$ holds. 
Owing to $\xi_1 \leq \beta \overset{\eqref{eq:eta-beta}}{\ll} \frac{\zeta}{d^4 \ln(\zeta^{-1})}$, we may write
$$
\ell_2\left(\zeta/2,d,(4\xi_1)^{-2}\right) \geq \ell_1\left(d+1,(4\xi_1)^{-2}\right). 
$$
Consequently, $H'$ contains a copy of $T_1$, namely $\mathcal{T}_1$, by Theorem~\ref{thm:AKS-rooted}. As $H'$ is rainbow with all colours used found in $\mathcal{A}_1$, so is $\mathcal{T}_1$. The elements of the set $Z_1 \subseteq V(T)$, that is, the roots associated with $T_1$, are embedded in $\mathcal{T}_1$; subsequent tree-embeddings whose roots have thus been determined are to be embedded in a {\sl rooted manner}.

At this stage, only the edges of $G[X_1]$ and their colours under $\psi$ have been fully exposed. Finally, note that no colour from the colour-reservoir set $\mathcal{R}$ is present on the edges of $\mathcal{T}_1$.

\bigskip

\noindent
{\bf Embedding $\boldsymbol{T_i}$.} Suppose that the trees $T_1, \ldots, T_{i-1}$ have all been embedded into $G[X_1] \cup \ldots \cup G[X_{i-1}]$ such that the following properties hold
\begin{itemize}
\item [(I.1)] $\mathscr{T} := \bigcup_{j=1}^{i-1} \mathcal{T}_j$ is rainbow under $\psi$, where $\mathcal{T}_j$ is the image of $T_j$ for every $j \in [i-1]$;

\item [(I.2)] $\mathscr{T} \cong \bigcup_{j=1}^{i-1} T_j$;

\item [(I.3)] No edge of $G$ with at least one endpoint in $X_i$ nor its colour have been exposed;

\item [(I.4)] $|\psi(E(\mathscr{T})) \cap \mathcal{R}| = O(1)$.

\end{itemize}

Write $\mathcal{R}_i$ to denote the subset of colours of $\mathcal{R}$ not appearing on the edges of $\mathscr{T}$, and note that $|\mathcal{R}_i| = \Omega(n)$ holds by (I.4). Let $\mathcal{A}_i = [n] \setminus (\psi(E(\mathscr{T})) \cup \mathcal{R})$ denote the set of {\sl available} colours for the embedding of $T_i$. Observe that 
\begin{equation}\label{eq:available}
|\mathcal{A}_i| \geq (\eps - \rho) n \overset{\eqref{eq:rho}}{\geq} \eps n/2.
\end{equation}
Finally, let $r \in V(\mathscr{T})$ be the predetermined root for the forthcoming embedding of $T_i$. Then, $r \notin X_i$ and, by (I.3), none of the pairs of the form $\{\{r,x\}: x \in X_i\}$ nor their colour have been thus far exposed.

\medskip

Write $\xi_i n := |X_i|$ so that $\xi_i \leq \beta$ holds. 
Since $G[X_i]~\sim \mathbb{G}(\xi_i n, \omega(1)/n)$, the event that an $n$-uniform colouring of $G[X_i]$ uses at least $\eps n /4$ colours from $\mathcal{A}_i$ occurs a.a.s. by Part~(a) of Lemma~\ref{lem:many-colours} and by~\eqref{eq:available}. Conditioning on this event, we randomly sparsify $G[X_i]$, as performed in the embedding argument of $T_1$, where in that argument we replace $X_1$ with $X_i$, $\mathcal{A}_1$ with $\mathcal{A}_i$, and $n/4$ with $\eps n/4$.

The resulting graph distribution $\mathbb{\tilde G}_i$ coincides with the uniform graph distribution, namely $\mathbb{G}(\xi_in, \eps n/4)$. Moreover, all graphs sampled according to this distribution are rainbow, with all colours used found in $\mathcal{A}_i$. 
Owing to $\xi_i \leq \beta \overset{\eqref{eq:eta-beta}}{\ll} \zeta \eps/d^2$, the inequality $\eps (4\xi_i)^{-2} \geq \max\left\{50 \zeta^{-1},\ell_1\left(d+2, \eps (4\xi_i)^{-2}\right)\right\}$ holds; this coupled with Lemma~\ref{lem:AKS-expander} (and aided by~\cite[Proposition~1.13]{JLR} as above) imply that a graph sampled from $\mathbb{\tilde G}_i$ is a.a.s. in 
$$
\mathrm{EXPAND}_2 := \mathrm{EXPAND}(\xi_in,\zeta/2, \eps (4\xi_i)^{-2}, (2d+1)^{-1}, d+2).
$$ 
Fix a graph $H \sim \mathbb{\tilde G}_i$ satisfying $\mathrm{EXPAND}_2$ and let $H' \subseteq H$ be its effective expander. In particular, $v(H') \geq (1 + 3\zeta/4) v(T_i)$ holds.

\medskip
Include each pair of $\{\{r,v\}: v \in V(H')\}$ as an edge independently at random with probability $\omega(1)/n$. Colour each such edge $rv$ uniformly at random from $[n]$. It follows by Part~(b) of Lemma~\ref{lem:many-colours} that there exists an edge-set $E_r \subseteq \{rv : v \in V(H')\}$ of size $\eps (4\xi_i)^{-2} \geq (d+2)^2$ which is rainbow and $\psi(e) \in \mathcal{R}_i$ for every $e \in E_r$. Fix such a set $E_r$. Lemma~\ref{lem:exp-degrade} then asserts that the graph $H'' := (V(H') \cup \{r\}, E(H') \cup E_r)$ is in 
$$
\mathrm{EXPAND}\left(v(H''), 0, \eps (4\xi_i)^{-2}, (2d+2)^{-1}, d+1\right)
$$ 
(that is, $H''$ is its own effective expander). Moreover, $H''$ is rainbow and $\psi(E(H'')) \cap \psi(E(\mathscr{T})) = \emptyset$. Since $\xi_i \leq \beta \overset{\eqref{eq:eta-beta}}{\ll} \frac{\zeta \eps}{d^4 \ln(\zeta^{-1})}$, it follows that 
$$
\ell_2\left(\zeta/2,d,\eps (4 \xi_i)^{-2}\right) \geq \ell_1\left(d+1,\eps (4 \xi_i)^{-2}\right).
$$ 
It thus follows, by Theorem~\ref{thm:AKS-rooted}, that a rainbow copy of $T_i$ rooted at $r$ exists in $H''$. 
The set of colours removed from the colour-reservoir set is a subset of $\{\psi(e) : e \in E_r\}$ and thus has size at most $O(1)$.  

It is now straightforward to verify that the, appropriately updated, properties (I.1) -- (I.4) are satisfied, thus concluding the proof.

\section{Prescribed rainbow spanning trees in randomly perturbed graphs}\label{sec:perturbed}

In this section we prove Theorem~\ref{th::rainbowTree}. The main ingredients of our proof are Theorem~\ref{thm:rainbow-AKS}, a {\sl randomness shift} argument, taken from~\cite{KKS17}, which {\sl shifts randomness}, so to speak, from the random perturbation to the seed, and an absorbing structure; the latter can be viewed as a rainbow variant of the one used in~\cite{HMMMO} (see also~\cite{BMPP18}).

\medskip
\noindent
{\bf Absorbers.} Let $H_1$ and $H_2$ be edge-disjoint graphs on the same vertex-set $[n]$ and set $H := H_1 \cup H_2$. Let $\psi : E(H) \to \mathbb{N}$ be an edge-colouring of $H$. Let $T$ be a tree with vertex-set $V(T) = [n]$. In order to avoid confusion, throughout this section, we refer to the members of $V(H)$ as {\sl vertices} and to the members of $V(T)$ as {\sl nodes}. Let $S \subseteq T$ be a subtree of $T$. Let $f : V(S) \to V(H_1)$ be an embedding of $S$ in $H_1$ such that $\mathcal{S} := f(S)$ is $\psi$-rainbow. Set $A := A(\mathcal{S}) := \{\psi(e) : e \in E(\mathcal{S})\}$ be the set of colours seen on the edges of $\mathcal{S}$ under $\psi$. Let $I := I(S) \subseteq V(S)$ be a set of nodes of $S$ such that $N_T(x) \subseteq V(S)$ for every $x \in I$. 
The members of $I$ are nodes of $T$ that are completely {\sl resolved} by $f$ in the sense that their entire closed neighbourhood in $T$ is embedded in $\mathcal{S}$. Set $\mathcal{I} := f(I)$. 

For two vertices $u, v \in V(H)$, let
$$ 
B(u,v) := \{x \in N_{H_2}(u) \cap \mathcal{I} : N_{\mathcal{S}}(x) \subseteq N_{H_2}(v)\}
$$ 
denote the set of vertices of $\mathcal{I}$ which are neighbours of $u$ in $H_2$ and, moreover, all their neighbours in the tree $\mathcal{S}$ are neighbours of $v$ in  $H_2$.  
Sets of the form $B(u,v)$ are used to extend a given rainbow embedding of a subtree of $T$, say $S$, into a rainbow embedding of another subtree of $T$, namely $\tilde S$, where $S \subsetneq \tilde S \subseteq T$. This is accomplished by {\sl absorbing} (in a rainbow fashion) a vertex of $H \setminus \mathcal{S}$; we denote the resulting tree by $\mathcal{\tilde S}$. 

The {\sl absorption}  of a vertex $v \in V(H) \setminus V(\mathcal{S})$ is performed in the following setting. Suppose that there exist nodes $u' \in V(S)$ and $v' \in V(T) \setminus V(S)$ such that $u'v' \in E(T)$; let $u := f(u')$, and  let $x \in B(u, v)$ be a vertex for which 
\begin{equation}\label{eq:colour-absorb}
\left|\left(\{\psi(ux)\} \cup \{\psi(yv) : y \in N_{\mathcal{S}}(x)\}\right) \setminus A\right| = 1 + |N_{\mathcal{S}}(x)|
\end{equation}
holds. Observe that $x \in \mathcal{I}$ holds by the definition of $B(u,v)$; in particular, the node $f^{-1}(x) \in V(S)$ is completely resolved by $f$. In order to absorb $v$ into $\mathcal{S}$ and consequently produce $\mathcal{\tilde S}$, we alter the image of $f^{-1}(x)$ by resetting it to $v$. This replacement is feasible since $N_{\mathcal{S}}(x) \subseteq N_{H_2}(v)$ by the definition of $B(u,v)$. Next, we embed $v'$, the neighbour of $u'$ in $T$, into $x$; this embedding is feasible since $ux \in E(H_2)$ by the definition of $B(u,v)$. More concisely, define $\tilde f$ to be given by 
$$
\tilde f(z) = 
\begin{cases}
v, & z = f^{-1}(x), \\
x, & z = v', \\
f(z), & \textrm{otherwise},
\end{cases}
$$
where $z \in V(S) \cup \{v'\}$. Then, $\tilde f$ is an embedding of the subtree of $T$ induced by $V(S) \cup \{v'\}$, namely $\tilde S$. We say that $x$ was \emph{used to absorb} $v$.

Assuming $\mathcal{S}$ is $\psi$-rainbow, it is seen that its extension $\tilde f$, defined above, is a $\psi$-rainbow embedding of $\tilde S$. Indeed, it follows by~\eqref{eq:colour-absorb} that $\psi(e) \neq \psi(e')$ for every two distinct edges $e \in E(\tilde{\mathcal S}) \setminus E({\mathcal S})$ and $e' \in E(\tilde{\mathcal S})$.

\medskip 
\noindent
{\bf Randomness shift.} Our proof of Theorem~\ref{th::rainbowTree} commences with  an application of Theorem~\ref{thm:rainbow-AKS} in order to find an embedding of a prescribed almost-spanning rainbow subtree of $T$, say $T'$, within the random perturbation and colouration thereof. The \emph{randomness shift} argument, described next, allows one to view this initial embedding of $T'$ as though it was sampled uniformly amongst all copies of $T'$ in $K_n$. This, in turn, allows one to appeal to the hypergeometric distribution in order to estimate the cardinality of certain $B(u,v)$-sets (see Claim~\ref{cl::LargeBuv} for details) that are then used for the sake of absorption.

Let $r$ be a positive integer, let $R \sim \mathbb{G}(n,p)$, and let $\psi : E(R) \to [r]$ be an edge-colouring of $R$. Suppose that a.a.s. $R$ contains a certain edge-coloured subgraph. As noted above, we may assume that this subgraph is uniformly distributed over its copies in $K_n$ (with an appropriate edge-colouring). Indeed, $R$ and $\psi$ can be generated as follows. First, generate a random graph $R' \sim \mathbb{G}(n,p)$ and colour its edges; denote the resulting colouring by $\psi'$. Next, permute the vertex-set of $R'$ {\sl randomly};  denote the resulting graph by $R$ and the resulting edge-colouring by $\psi$. That is, choose a permutation $\pi \in S_n$ uniformly at random and set $R = ([n], \{\pi(u) \pi(v) : uv \in E(R')\})$ and $\psi(\pi(u) \pi(v)) = \psi'(uv)$ for every $uv \in E(R')$. The corresponding probability space coincides with $\mathbb{G}(n,p)$ with an appropriate edge-colouring; in particular $G' \subseteq R'$ is rainbow under $\psi'$ if and only if $\pi(G') \subseteq R$ is rainbow under $\psi$. In this manner, the aforementioned edge-coloured subgraph of $R$ is sampled uniformly at random through $\pi$. 

\medskip

We are now ready to prove Theorem~\ref{th::rainbowTree}.

\begin{proof} [Proof of Theorem~\ref{th::rainbowTree}]
Let $\delta$, $d$, $\alpha$ and $T$ be as in the premise of the theorem. Set 
$$
\eps = \left(\frac{\delta}{4 d} \right)^{d+1} \cdot \frac{1}{10 d^2}.
$$ 
Let $T_0$ be a tree on $\ell := (1 - \eps) n$ (which we assume is an integer) vertices which is obtained by successively removing leaves from $T$. 
Let $G \in \mathcal{G}_{\delta,n}$ and let $R \sim \mathbb{G}(n, \omega(1)/n)$; we may assume that the $\omega(1)$ term is sufficiently small so that a.a.s. $\Delta(R) = o(\ln n)$. Let $\psi$ be a $[(1 + \alpha) n]$-uniform colouring of $G \cup R$. 

It follows by Theorem~\ref{thm:rainbow-AKS} that $R$ a.a.s. admits a $\psi$-rainbow copy of $T_0$. Fix $R \sim \mathbb{G}(n, \omega(1)/n)$ and an $n$-uniform colouring $\psi$ of its edges such that $\Delta(R) = o(\ln n)$ and such that there exists an embedding $f : V(T_0) \to [n]$ for which $\mathcal{T}_0 := f(T_0)$ is $\psi$-rainbow. We proceed to use the edges of $E(G) \sm E(R)$ in order to absorb all vertices in $[n] \sm V(\mathcal{T}_0)$ in a series of (rainbow) absorption steps.

Owing to $\Delta(R) = o(\ln n)$, the graph $G \setminus E(R)$ is a member of $\mathcal{G}_{0.9 \delta, n}$. Let $H_1, \ldots, H_d$ be spanning edge-disjoint subgraphs of $G \setminus E(R)$ such that $\delta(H_i) \geq \delta n/(2 d)$ holds for every $1 \leq i \leq d$.
A standard argument shows that an appropriately defined random partition of $E(G)\sm E(R)$ into $d$ parts yields the aforementioned decomposition a.a.s. and so, in particular, the desired graphs $H_1, \ldots, H_d$ exist. The role of this partition is to ensure that there are edges which are needed for the absorption of a vertex, whose colour has not yet been exposed. Indeed, suppose that $u'$ is some vertex of $T$ that was already embedded, and $v'_1, \ldots, v'_r$ are its neighbours in $T$ that were not yet embedded. Since, clearly, $r \leq d$, we will be able to ensure that when we wish to embed $v'_j$ for some $1 \leq j \leq r$, there will be an $1 \leq i \leq d$ for which the colours of the edges of $H_i$ that are incident with the image of $u'$ were not previously exposed.    

Let $I_0 \subseteq V(T_0)$ be a set satisfying the following properties.
\begin{itemize}
\item [(Q.1)] $\textrm{dist}_{T_0}(x,y) \geq 3$ for any two distinct vertices $x, y \in I_0$;
\item [(Q.2)] $N_T(x) \subseteq V(T_0)$ for every $x \in I_0$;
\item [(Q.3)] $|I_0| = \left\lfloor \frac{n - (d + 1) \eps n}{d^2 + 1} \right\rfloor$.
\end{itemize}
Such a set $I_0$ can be constructed via a simple greedy procedure. Let $\mathcal{I}_0 = f(I_0)$. For every $1 \leq j \leq d$ and distinct vertices $u,v \in [n]$, set 
$$
B_j(u,v) := \{x \in N_{H_j}(u) \cap \mathcal{I}_0 : N_{\mathcal{T}_0}(x) \subseteq N_{H_j}(v)\}.
$$ 
We prove that a.a.s. the set $B_j(u,v)$ is large for every $j \in [d]$ and $u, v \in [n]$; this is done without exposing the colours of the edges of $G \setminus E(R)$.

\begin{claim} \label{cl::LargeBuv}
Asymptotically almost surely $|B_j(u,v)| \geq \left(\frac{\delta}{4 d} \right)^{d+1} \cdot \frac{n}{5 d^2}$ holds for every $1 \leq j \leq d$ and every $u, v \in [n]$.
\end{claim}

\begin{proof}
Let $u, v \in [n]$ and $k \in [d]$ be fixed. As explained in the randomness shift argument, we may assume that a random permutation $\pi : V(R) \to [n]$ maps $\mathcal{T}_0$ to an isomorphic tree; in particular, the source of randomness throughout the proof of the claim is the location of $\mathcal{T}_0$. Let $w_1, w_2, \ldots, w_n$ be an ordering of $V(R)$ satisfying the following properties. 
\begin{itemize}
\item [(P.1)] $V(\mathcal{T}_0) = \{w_1, \ldots, w_{\ell}\}$;
\item [(P.2)] for every positive integer $i$, if $w_i \in \mathcal{I}_0$, then $w_{i+j} \in N_{\mathcal{T}_0}(w_i)$ for every $1 \leq j \leq \deg_{\mathcal{T}_0}(w_i)$;
\item [(P.3)] for all positive integers $i$ and $j$, if $w_i \in \mathcal{I}_0 \cup N_{\mathcal{T}_0}(\mathcal{I}_0)$ and $w_j \notin \mathcal{I}_0 \cup N_{\mathcal{T}_0}(\mathcal{I}_0)$, then $i < j$.
\end{itemize}
Owing to properties~(Q.1) and~(Q.2) stated above, such an ordering exists. We may assume that the images $\pi(w_1), \pi(w_2), \ldots, \pi(w_{\ell})$ are determined (randomly) first and in this order; then the images $\pi(w_j)$ are set for every $\ell + 1 \leq j \leq n$ in an arbitrary order. 

For $i \in [\ell]$, let $X_i$ denote the indicator random variable for the event that $\pi(w_i) \in N_{H_k}(u)$; let $Y_i$ denote the indicator random variable for the event that $\pi(w_i) \in N_{H_k}(v)$. For every index $i$ such that $w_i \in \mathcal{I}_0$, set 
$$
Z_i := X_i \cdot \prod_{j=1}^{\deg_{\mathcal{T}_0}(w_i)} Y_{i+j}.
$$ 
By Property (P.2) stated above, the random variable $Z_i$ is the indicator random variable for the event that $\pi(w_i) \in B_k(u,v)$. We may then write that $|B_k(u,v)| = \sum Z_i$, where the sum ranges over all $i \in [\ell]$ for which $w_i \in \mathcal{I}_0$. 

Let $A_u$ (respectively, $A_v$) denote the event that 
$$
|N_{H_k}(u) \cap \{\pi(w_1), \ldots, \pi(w_{n/3})\}| \geq |N_{H_k}(u)|/2
$$ 
(respectively, $|N_{H_k}(v) \cap \{\pi(w_1), \ldots, \pi(w_{n/3})\}| \geq |N_{H_k}(v)|/2$). The random variable $|N_{H_k}(u) \cap \{\pi(w_1), \ldots, \pi(w_{n/3})\}|$ (and its counterpart for $v$) is distributed hypergeometrically owing to the randomness shift argument. A straightforward application of Chernoff's bound for the hypergeometric distribution then yields that $\mathbb{P}(A_u \cup A_v) = o(1)$. Hence, throughout the remainder of the proof we assume that $A^c_u \cap A^c_v$ holds.

Recalling that $\delta(H_k) \geq \delta n/(2d)$, we may  write that 
$$
\mathbb{P}(Z_i = 1) \geq \frac{|N_{H_k}(u)| - \sum_{j=1}^{i-1} X_j}{n} \cdot \prod_{t=1}^{\deg_{\mathcal{T}_0}(w_i)} \frac{|N_{H_k}(v)| - \sum_{j=1}^{i+t-1} Y_j}{n} \geq \left(\frac{\delta}{4 d}\right)^{d+1}
$$
holds for every $ i \in [n/3 - d]$ for which $w_i \in \mathcal{I}_0$. Properties~(P.2) and~(P.3) imply that 
$$
|\mathcal{I}_0 \cap \{w_1, \ldots, w_{n/3 - d}\}| \geq \min \left\{|\mathcal{I}_0|, \frac{n - 3d}{3(d+1)} \right\} \geq |\mathcal{I}_0|/2
$$ 
holds, where the last inequality is owing to Property~(Q.3) and the assumption that $d \geq 2$.  
Therefore, even though the random variables $Z_i$, defined above, are not necessarily mutually independent, we may still write  
\begin{align*}
\mathbb{P} \left(|B_k(u,v)| < \left(\frac{\delta}{4 d} \right)^{d+1} \cdot \frac{n}{5 d^2} \right) &\leq \mathbb{P} \left(\Bin \left(\frac{|\mathcal{I}_0|}{2}, \left(\frac{\delta}{4 d} \right)^{d+1} \right) < \left(\frac{\delta}{4 d} \right)^{d+1} \cdot \frac{n}{5 d^2} \right) \\
&< e^{- \Omega \left(\left(\frac{\delta}{4 d} \right)^{d+1} \cdot \frac{n}{d^2} \right)},
\end{align*}
where the last inequality follows by a standard application of Chernoff's bound. 

To conclude the proof, a union bound over every $k \in [d]$ and all pairs $u, v \in [n]$, shows that the probability that there exist such an index $k$ and a pair of vertices $u, v \in [n]$ for which $|B_k(u,v)| < \left(\frac{\delta}{4 d} \right)^{d+1} \cdot \frac{n}{5 d^2}$ holds, is $o(1)$.
\end{proof}

As mentioned above, the source of randomness underlying Claim~\ref{cl::LargeBuv} is the location of $\mathcal{T}_0$. Fix this tree such that the assertion of Claim~\ref{cl::LargeBuv} holds. We proceed with the absorption argument for the leftover vertices found in $[n] \setminus V(\mathcal{T}_0)$; let  $v_1, \ldots, v_r$ be an arbitrary enumeration of these vertices. 
Throughout the absorption process, a nested sequence of rainbow trees, namely $\mathcal{T}_0, \ldots, \mathcal{T}_r$, in $G \cup R$ is defined, where $V(\mathcal{T}_i) = V(\mathcal{T}_0) \cup \{v_1, \ldots, v_i\}$. This sequence of trees in $G \cup R$ corresponds to a nested sequence $T_0, \ldots, T_r = T$ of subtrees of $T$ such that $\mathcal{T}_i \cong T_i$ for every $0 \leq i \leq r$.

Suppose that, for some $0 \leq i < r$, we have already defined $T_0, \ldots, T_i$ and their respective images $\mathcal{T}_0, \ldots, \mathcal{T}_i$, and now wish to define $\mathcal{T}_{i+1}$ by absorbing $v_{i+1}$ in a rainbow fashion. Fix some $u \in V(\mathcal{T}_i)$ for which there exists a vertex $v' \in V(T) \setminus V(T_i)$ such that $u'v' \in E(T)$, where $u' := f^{-1}(u)$. For every $k \in [i]$, let $x_k \in \mathcal{I}_0$ denote the vertex that was used to absorb $v_k$ and set  
$$
B_{j,i}(u,v_{i+1}) := B_j(u,v_{i+1}) \setminus \{x_1, \ldots, x_i\}
$$ 
for every $j \in [d]$. The elements of $B_{j,i}(u,v_{i+1})$ are the potential absorbers for $v_{i+1}$ with respect to $u$. In particular, $B_{j,0}(u,v_1) = B_j(u,v_1)$. As the absorption of additional vertices proceeds, these sets diminish in size; nevertheless, this decrease may be controlled as 
\begin{equation} \label{eq:Bi}
|B_{j,i}(u,v_{i+1})| \geq |B_j(u,v_{i+1})| - i \geq \left(\frac{\delta}{4 d} \right)^{d+1} \cdot \frac{n}{5 d^2} - \eps n \geq \left(\frac{\delta}{4 d} \right)^{d+1} \cdot \frac{n}{10 d^2}
\end{equation}
holds for every $0 \leq i \leq r-1$ and every $j \in [d]$, by Claim~\ref{cl::LargeBuv} and by our choice of $\eps$.

The removal of vertices that were previously used for absorption is crucial in two respects. First, absorbing vertices cannot be reused. Second, there is a need to keep track over those edges for which the random colouring $\psi$ has been exposed. The manner by which the colouring is exposed throughout the absorption process is as follows. Let $j^* \in [d]$ be the smallest integer $k \in [d]$ for which the colours of the members of $E(H_k) \setminus E(R)$ incident to $u$ are not yet exposed. Such an index $j^*$ exists since, as can be seen below, colours of the edges in $E(G) \setminus  E(R)$ which are incident with $u$, are only revealed upon embedding $u$ or one of its neighbours in the tree; moreover, either $u \in V(\mathcal{T}_0)$ or $u \in V(\mathcal{T}_s)$ for some $s \in [i]$ and one of its neighbours is in $V(\mathcal{T}_{s-1})$. For every vertex $x \in B_{j^*,i}(u,v_{i+1})$, expose the colours of the edges in 
$$
E(H_{j^*}) \cap \left(\{u x\} \cup \{y v_{i+1} : y \in N_{\mathcal{T}_0}(x)\} \right).
$$ 
Note crucially that the colours of $E_{H_{j^*}}(v_{i+1}, \mathcal{T}_i) \cup E_{H_{j^*}}(u, \mathcal{I}_0)$ were not previously exposed. In particular, $E(H_1 \cup \ldots \cup H_d) \cap E(R) = \emptyset$ and thus the exposure of $\psi$ along $E(R)$ is of no consequence here. 

To conclude the absorption argument and, indeed, the proof of Theorem~\ref{th::rainbowTree}, observe that if there exists a vertex $x \in B_{j^*,i}(u,v_{i+1})$ such that 
$$
|(\{\psi(ux)\} \cup \{\psi(y v_{i+1}) : y \in N_{\mathcal{T}_0}(x)\}) \setminus \{\psi(e) : e \in E(\mathcal{T}_i)\}| = 1 + |N_{\mathcal{T}_0}(x)|,
$$ 
then it can be used to absorb $v_{i+1}$ as explained above. It follows by~\eqref{eq:Bi} that the probability that no such vertex exists is at most
$$
\left(1 - \left(\frac{\alpha}{1 + \alpha} \right)^{d+1}\right)^{|B_{j^*,i}(u,v_{i+1})|} \leq e^{- \left(\frac{\alpha}{1 + \alpha}\right)^{d+1} \left(\frac{\delta}{4 d} \right)^{d+1} \cdot \frac{n}{10 d^2}} = o(1/n).
$$    
This holds for every $0 \leq i \leq r-1$; consequently, a union bound shows that the probability that the absorption of any of $v_1, \ldots, v_r$ fails is $o(1)$.
\end{proof}

\section{Rainbow spanning trees in randomly perturbed graphs}

In this section, we prove Proposition~\ref{th::uarColouring}. 
The following two results will be used in our proof.

\begin{lemma} [\cite{BFKM04}] \label{lem::highConnectivity}
Let $H = (V,E)$ be a graph on $n$ vertices with minimum degree $k > 0$. Then, there exists a partition $V = V_1 \discup \ldots \discup V_t$ such that, for every $i \in [t]$, the subgraph $H[V_i]$ is $k^2/(16 n)$-connected and $|V_i| \geq k/8$. 
\end{lemma}
%
%\begin{theorem} [\cite{NashWilliams, Tutte}] \label{th::NashWilliamsTutte}
%A graph admits $k$ pairwise edge-disjoint spanning trees if and only if for every partition $S$ of its vertex set there are at least $k(|S|-1)$ edges between the parts of $S$.
%\end{theorem}
%
%\begin{corollary} [\cite{NashWilliams, Tutte}] \label{cor::NashWilliamsTutte}
%Every $2k$-edge-connected graph admits $k$ pairwise edge-disjoint spanning trees.
%\end{corollary}

\begin{theorem} [\cite{Suzuki}] \label{th::rainbowTreeCriterion}
Let $G$ be an $n$-vertex edge-coloured graph. Then, $G$ admits a rainbow spanning tree if and only if, for every $2 \leq s \leq n$ and for every partition $S$ of $V(G)$ into $s$ parts, there are at least $s-1$ edges between the parts of $S$, all assigned different colours.
\end{theorem}

\begin{proof} [Proof of Proposition~\ref{th::uarColouring}]
Let $\Pi:= V_1 \discup \ldots \discup V_t$ be a partition of $V(G)$ such that for every $i \in [t]$ it holds that $|V_i| \geq \delta n/8$ and $G[V_i]$ is $\delta^2 n/16$-connected; such a partition exists by Lemma~\ref{lem::highConnectivity}. Let $G' \sim G \cup \mathbb{G}(n,p)$, where $p = \omega(n^{-2})$. A standard argument shows that a.a.s. $e_{G'} (V_i, V_j) \geq 2t$ holds for every $1 \leq i < j \leq t$; for the remainder of the proof, we assume that $G'$ satisfies this property. 

\medskip

Let $\psi$ be an $(n-1)$-uniform colouring of $G'$. To complete the proof of Proposition~\ref{th::uarColouring}, we prove that $(G',\psi)$ a.a.s. satisfies the sufficient (and necessary) condition for the existence of a rainbow spanning tree stated in Theorem~\ref{th::rainbowTreeCriterion}. 
To that end, we distinguish between the following four types of partitions $\Pi' := U_1 \discup \ldots \discup U_s$ of $V(G')$.
\begin{description} 
\item [Case I:] Consider the case where the partition $\Pi'$ does not {\sl refine}  any part of the partition $\Pi$. That is, for every $i \in [t]$ there exists an index $j \in [s]$ such that $V_i \subseteq U_j$. The probability that there exists such a partition with at most $s-2$ distinct colours on the edges connecting its parts is at most
$$
t^t \binom{n-1}{s-2} \left(\frac{s-2}{n-1} \right)^{2t} \leq \left(\frac{t^3}{n-1} \right)^t = o(1).
$$  
Indeed, $s \leq t$ so that there are at most $t^t$ such partitions; moreover, for $G'$ we have that $\sum_{1 \leq i < j \leq s} e_{G'}(U_i, U_j) \geq 2t$, as stated above. 

\item [Case II:] Suppose that $2 \leq s \leq \ln n$. Owing to Case~I, we may assume that $\Pi'$ refines some part of $\Pi$; that is, there are indices $k \in [t]$ and $1 \leq i < j \leq s$ such that $V_k \cap U_i \neq \emptyset$ and $V_k \cap U_j \neq \emptyset$. Since $G[V_k]$ is $\delta^2 n/16$-connected and the partition $\Pi'$ refines $V_k$, it follows that $\sum_{1 \leq i < j \leq s} e_{G'}(U_i, U_j) \geq \delta^2 n/16$. Therefore, the probability that there exists such a partition with at most $s-2$ distinct colours on the edges connecting its parts is at most
$$
\sum_{s = 2}^{\ln n} s^n \binom{n-1}{s-2} \left(\frac{s-2}{n-1} \right)^{\delta^2 n/16} \leq \sum_{s = 2}^{\ln n} e^{n \ln s + s \ln n - \frac{\delta^2 n}{16} \cdot \ln \left(\frac{n-1}{s-2} \right)} = o(1).
$$

\item [Case III:] Suppose, next, that $\ln n \leq s \leq n - 20 \delta^{-1} \ln n$. For $i \in [s]$, the part $U_i$ is said to be {\sl small} if $|U_i| \leq \delta n/2$ and {\sl large} otherwise; clearly there are at most $2/\delta$ large parts. 
This simple observation and the assumption that $s \geq \ln n$ jointly imply that there are at least $s/2$ small parts. Assume, without loss of generality, that $U_1, \ldots, U_{s/2}$ are all small. For $i \in [s/2]$, let $u_i \in U_i$ be an arbitrary vertex. 
Note that
$$
\deg_{G'}(u_i, V(G') \setminus U_i) \geq \deg_G(u_i, V(G') \setminus U_i) \geq \delta(G) - |U_i| \geq \delta n/2
$$
holds for every $i \leq [s/2]$. Hence, 
$$
\sum_{1 \leq i < j \leq s/2} e_{G'}(U_i, U_j) \geq \frac{s}{2} \cdot \frac{\delta n}{2} \cdot \frac{1}{2} = \delta n s/8.
$$ 
Therefore, the probability that there exists such a partition with at most $s-2$ distinct colours on the edges connecting its parts is at most
$$
\sum_{s = \ln n}^{n - 20 \delta^{-1} \ln n} s^n \binom{n-1}{s-2} \left(\frac{s-2}{n-1} \right)^{\delta n s/8} \leq \sum_{s = \ln n}^{n - 20 \delta^{-1} \ln n} e^{n \ln s + s \ln n - \frac{\delta n s}{8} \cdot \ln \left(\frac{n-1}{s-2} \right)} = o(1),
$$  
where the equality holds by the assumed upper bound on $s$ and the known fact that $\ln (1+x) \approx x$ whenever $x$ tends to $0$.

\item [Case IV:] Suppose, lastly, that $n - 20 \delta^{-1} \ln n \leq s \leq n$. The same argument as in Case~III shows that 
$$
\sum_{1 \leq i < j \leq s} e_{G'}(U_i, U_j) \geq \delta n s/8.
$$ 
Moreover, the assumed lower bound on $s$ implies that the number of parts whose size is at least $2$ is at most $20 \delta^{-1} \ln n$ and that the total number of vertices in these parts is at most $40 \delta^{-1} \ln n$. Therefore, for any $n - 20 \delta^{-1} \ln n \leq s \leq n$, the number of partitions $\Pi' = U_1 \discup \ldots \discup U_s$ is at most 
$$
\binom{n}{40 \delta^{-1} \ln n} \left(40 \delta^{-1} \ln n \right)^{40 \delta^{-1} \ln n} \leq \left(\frac{e n}{40 \delta^{-1} \ln n} \cdot 40 \delta^{-1} \ln n \right)^{40 \delta^{-1} \ln n} \leq n^{50 \delta^{-1} \ln n}.
$$
We conclude that the probability that there exists such a partition with at most $s-2$ distinct colours on the edges connecting its parts is at most 
$$
\sum_{s = n - 20 \delta^{-1} \ln n}^n n^{50 \delta^{-1} \ln n} \binom{n-1}{s-2} \left(\frac{s-2}{n-1} \right)^{\delta n s/8} \leq \sum_{s = n - 20 \delta^{-1} \ln n}^n e^{100 \delta^{-1} (\ln n)^2 - \frac{\delta n s}{8} \cdot \ln \left(\frac{n-1}{s-2} \right)} = o(1),
$$
where the inequality holds by the assumed lower bound on $s$ and the equality holds by the assumed upper bound on $s$ and the known fact that $\ln (1+x) \approx x$ whenever $x$ tends to $0$. 
\end{description} 
\end{proof}

\section*{Acknowledgements} We thank Joshua Erde for drawing our attention to several known results that are related to Proposition~\ref{th::uarColouring}.

\bibliographystyle{amsplain}
\bibliography{lit}

\appendix

\section{Proof of Lemma~\ref{lem:AKS-expander}}\label{app:AKS-expander}\label{sec:app-AKS-expand}

Lemma~\ref{lem:AKS-expander} is implied by the following result. Our proof of this result is that of~\cite[Lemma~3.1]{AKS07} with the necessary modifications to some of the constants, made in order to facilitate the slightly larger range of values of $\eta$ seen in Lemma~\ref{lem:AKS-expander}.   

\begin{lemma}
For every $0 < \theta < 1/2$, $C \geq 50/\theta$, integer $r \geq 3$, real $0 < \eta \leq 1/(r+2)$, and integer $k \geq 2 e^4 r$, the random graph $G \sim \mathbb{G}(n, 4C/n)$ a.a.s. has an induced subgraph $H$ satisfying the following properties. 
\begin{itemize}
	\item [{\em (A.1)}] $v(H) \geq (1 - \theta)n$. 
	\item [{\em (A.2)}] $C \leq \delta(H) \leq \Delta(H) \leq 10 C$.
	\item [{\em (A.3)}] Every vertex-induced subgraph $H' \subseteq H$ 
	with $\delta(H') \geq k r \ln(C)$ is an $(\eta,r)$-expander. 
\end{itemize}
\end{lemma}

\noindent
{\sl Remark.} Property~(A.3) could be meaningless if no vertex-induced subgraph of $H$ has minimum degree at least $kr \ln(C)$.   

\begin{proof}
The existence of an induced subgraph $H$ of $G \sim \mathbb{G}(n,4C/n)$ satisfying properties~(A.1) and~(A.2) is ensured by the arguments seen in~\cite[Proof of Lemma~3.1]{AKS07} (hence, we do not repeat them here). It is in this part of the proof found in~\cite{AKS07} that the condition $C \geq 50/\theta$, seen in the premise of the lemma, is used. Our arguments below pertain solely to the expansion properties of $H$, as stated in Property~(A.3). Throughout these arguments, no additional constraints are imposed on $C$. 

\medskip

Let $\mathcal{E}$ denote the event that there exists a subset $U \subseteq [n]$ such that $\delta(H[U]) \geq kr \ln(C)$ and yet $H[U]$ is not an $(\eta,r)$-expander. Our aim is to prove that $\Pr[\mathcal{E}] = o(1)$. In order to do so, we first define the following auxiliary events.
\begin{itemize}
	\item [$A_t$:] there exist disjoint vertex-sets $X$ and $Y$ satisfying $|X| = t$ and $|Y| = rt$ such that $e_G(X \cup Y) \geq \ell t$, where $\ell := kr \ln(C)/2$. 
	
	\item [$B_t$:] there exist disjoint vertex-sets $X$ and $Z$ satisfying $|X| = t$ and $|Z| = t$ such that $e_G(X,Z) = 0$. 
\end{itemize}

Observe that if $\mathcal{E}$ occurs, that is, if there exists a set $U \subseteq [n]$ such that $\delta(H[U]) \geq kr \ln(C)$ and yet $H[U]$ is not an $(\eta,r)$-expander, then there exists a set $X \subseteq U$ of size $t$ for some $1 \leq t \leq \eta |U| \leq \eta n$ such that $|\Gamma_{G[U]}(X)| < rt$ (the latter inequality holds since $H$ is an induced subgraph of $G$). Since $\delta(H[U]) \geq kr \ln(C)$ by assumption, it follows that $e_G(X \cup \Gamma_{G[U]}(X)) \geq t kr \ln(C)/2$, which is precisely the event $A_t$. Similarly, $e_G(X, U \setminus (X \cup \Gamma_{G[U]}(X))) = 0$ and 
$$ %\begin{equation} \label{eq::Bt}
|U| - |X| - |\Gamma_{G[U]}(X)| \geq |U| - \eta |U| - r t 
\geq (1-\eta)\frac{t}{\eta} - rt = \left(1/\eta - 1 - r\right) t \geq t,
$$ %\end{equation}
where the last inequality holds since $\eta \leq 1/(r+2)$; this is precisely the event $B_t$. We conclude that 
$$
\Pr[\mathcal{E}] \leq \sum_{t=1}^{\eta n} \Pr[\mathcal{A}_t \wedge \mathcal{B}_t] \leq \sum_{t=1}^{\gamma n} \Pr[A_t] + \sum_{t = \gamma n}^{\eta n} \Pr[B_t],
$$
where $\gamma := \ln (C)/C$.

%\noindent
%In both events, namely $A_t$ and $B_t$, the fact that $H$ is a vertex-induced subgraph of $G$ allows for these events to be phrased using $G$ instead of $H$. 
%Regarding the bound on $|Z|$ seen in $B_t$, note that if $H[U]$ is not an $(\eta,r)$-expander for some (nonempty) $U \subseteq [n]$, then there exists a subset $X \subseteq U$ of size $1 \leq t \leq \eta |U|$ for which $|\Gamma_{G[U]}(X)|< rt$ and $e_G(X,U \setminus (X \cup \Gamma_{G[U]}(X))) = 0$ both hold. Then,
%$$
%|U| - |X| - |\Gamma_{G[U]}(X)| \geq |U| - \eta |U| - r t 
%\geq (1-\eta)\frac{t}{\eta} - rt = \left(1/\eta - r - 1\right)t \geq t.
%$$
%holds, where the last equality is owing to $\eta \leq 1/(r+2)$.

\medskip
It is therefore sufficient to prove that $\sum_{t=1}^{\gamma n} \Pr[A_t] = o(1)$ and that $\sum_{t = \gamma n}^{\eta n} \Pr[B_t] = o(1)$. Starting with the former, note that  
\begin{align*}
\Pr[A_t] & \leq \binom{n}{t}\binom{n}{rt}\binom{(t + rt)^2}{\ell t}\left(\frac{4C}{n}\right)^{\ell t}\\
& \leq \left(\frac{en}{t} \left(\frac{en}{rt}\right)^r \left(\frac{e^4rt}{k \ln (C)}\right)^{\ell} \left(\frac{C}{n}\right)^{\ell} \right)^t\\
& \leq \left(\exp(-\ell/4) \left(\frac{n}{t}\right)^{r+1}\left(\frac{C t}{n \ln (C)}\right)^{\ell} \right)^t\\
& = \left(C^{-kr/8} \left(\frac{n}{t}\right)^{r+1}\left(\frac{t}{\gamma n }\right)^{\ell} \right)^t,
\end{align*}
where the penultimate inequality holds since $(e/r)^r < 1$ holds by our assumption $r \geq 3$, and $\ln(e^4r/k) < - 1/2$ holds by our assumption $k \geq 2 e^4 r$. Since $n^{r+1 - \ell} t^{\ell} = o(n^{-1})$ holds for every $1 \leq t \leq 2\ln n$, it follows that $\sum_{t=1}^{2\ln n} \Pr[A_t] = o(1)$. Moreover, for every $1 \leq t \leq \gamma n$, we have 
$$
C^{-kr/8} \left(\frac{n}{t}\right)^{r+1}\left(\frac{t}{\gamma n}\right)^{\ell} \leq C^{- k r/8 + r + 1} \leq e^{-1}.
$$ 
Hence, $\Pr[A_t]  \leq e^{- 2\ln n} = o(n^{-1})$, whenever $2\ln n \leq t \leq \gamma n$; consequently, $\sum_{t = 2\ln n}^{\gamma n} \Pr[A_t] = o(1)$. We conclude that $\sum_{t=1}^{\gamma n} \Pr[A_t] = o(1)$ holds, as required.

\medskip
We proceed to proving that $\sum_{t = \gamma n}^{\eta n} \Pr[B_t] = o(1)$. For every $\gamma n \leq t \leq \eta n$, it holds that 
\begin{align*}
\Pr[B_t] & \leq \binom{n}{t}^2 \left(1-\frac{4C}{n}\right)^{t^2} \leq \left( \left(\frac{en}{t}\right)^{2} \exp(-4C t/n) \right)^{t} \\
& \leq \left( \left(\frac{e  C}{\ln (C)}\right)^2 C^{-4}\right)^t \overset{\phantom{\rho \geq 1}}{\leq}  C^{-2 t} = o(n^{-1}). 
\end{align*}
Hence, $\sum_{t = \gamma n}^{\eta n} \Pr[B_t] = o(1)$; concluding the proof.      
\end{proof}

\end{document}